\documentclass[reqno, 11pt]{amsart}
\usepackage{graphicx, enumerate, amsmath, amssymb,amsthm,manfnt,hyperref,amscd,braket}
\usepackage[utf8]{inputenc} % set input encoding (not needed with XeLaTeX)
 \allowdisplaybreaks

\newcommand{\norm}[1]{\left\Vert#1\right\Vert}
\newcommand{\abs}[1]{\left\vert#1\right\vert}

\newcommand{\D}{\mathbb{D}}
\newcommand{\rl}{{\mathbb{R}}}
\newcommand{\cx}{{\mathbb{C}}}

\newcommand{\dbar}{\overline{\partial}}

\theoremstyle{plain}
\newtheorem{theorem}{Theorem}[section]
\newtheorem{proposition}{Proposition}[section]
\newtheorem{lemma}{Lemma}[section]
\newtheorem{corollary}{Corollary}[section]

\theoremstyle{remark}

\newcommand{\Bp}{\mathbf{B}}
\newcommand{\Ht}{\mathbb{H}}
\newcommand{\BpHt}{\mathbf{B}_{\mathbb{H}}}

\newcommand{\Dd}{{\mathbb{D}^{\ast}\times\mathbb{D}}}
\newcommand{\BpDd}{\mathbf{B}_{\mathbb{D}^{\ast}\times\mathbb{D}}}

\title[Bergman Projection on the Hartogs Triangle]{$L^p$ Mapping Properties of the Bergman Projection on the Hartogs Triangle}
\author{Debraj Chakrabarti}
\address[Debraj Chakrabarti]{Central Michigan University, Department of
Mathematics, Mt. Pleasant, MI 48859}
\email{chakr2d@cmich.edu }

\author{Yunus E. Zeytuncu}
\address[Yunus E. Zeytuncu]{University of Michigan - Dearborn, Department of Mathematics and Statistics, Dearborn, MI 48128}
\email{zeytuncu@umich.edu}
\date{}
\thanks{ This work was partially supported by a grant from the Simons Foundation (\#316632 to Debraj Chakrabarti), and also by an Early Career internal grant from Central Michigan University to Debraj Chakrabarti.}
\begin{document}

\maketitle
\begin{abstract} We prove optimal estimates for the mapping properties  of the  Bergman projection on the  Hartogs 
triangle in weighted $L^p$ spaces when $p>\frac{4}{3}$, where the weight is a power of the distance to the singular boundary point. For $1<p\leq\frac{4}{3}$ we show that no such weighted estimates are possible.
\end{abstract}

\section{Introduction}
\subsection{Results}
In this note we  describe the $L^p$ regularity of the Bergman projection on the {\em Hartogs triangle}, the pseudoconvex
domain in $\cx^2$ defined as
\begin{equation}\label{eq-ht}
 \Ht=\Set{(z_1,z_2)\in \cx^2 |\abs{z_2}<\abs{z_1}<1},
\end{equation} for $1<p<\infty$. To state our results we use {\em weighted Bergman spaces} $A^p(\Ht,\omega)$, where
$\omega>0$ is a continuous function on $\Ht$, and a holomorphic function $f\in \mathcal{O}(\Ht)$ belongs to $A^p(\Ht,\omega)$ 
if $\norm{f}_{L^p(\Ht,\omega)}^p= \int_\Ht \abs{f}^p \omega dV <\infty,$
where $dV$ is Lebesgue measure of $\cx^2$.  Let $\delta_1$ be defined on $\cx^2$ by  
 \begin{equation}\label{eq-delta1} \delta_1(z_1,z_2)= \abs{z_1}.\end{equation}
 On $\Ht$, we have $ \frac{\phantom{\sqrt{}}1}{\sqrt{2}}\abs{z}\leq \delta_1(z) \leq \abs{z}$, i.e.     $\delta_1$ is comparable to the distance to the singular point 0. Let $\BpHt:L^2(\Ht)\to A^2(\Ht)$ denote the Bergman projection.
 We begin with the case $p\geq 2$:
\begin{theorem}\label{thm-bdd} If $p\geq 2$, the map $\BpHt$ is bounded  and surjective  from $L^p(\Ht)$ to $A^p(\Ht, \delta_1^{p-2})$.
\end{theorem}
Note that the surjectivity of $\BpHt$ means that we have  the best possible 
estimates.
From Theorem~\ref{thm-bdd}  we  can recover the following folk result: 
\begin{corollary} \label{cor-folk}
(a)  If $\frac{4}{3}< p <{4}$, then $\BpHt$ is bounded (and surjective) from $L^p(\Ht)$ to $A^p(\Ht)$.

(b) If $p\geq 4$, then $\BpHt$ does not map $L^p(\Ht)$ into  $A^p(\Ht)$.
\end{corollary}
See  \cite{Chen13} for a generalization of part (a) to a class of domains closely related to $\Ht$.
Note that for $2\leq p < 4$, both Theorem~\ref{thm-bdd} and Corollary~\ref{cor-folk} apply.  This is because in 
this range the space $A^p(\Ht,\delta_1^{p-2})$ coincides with $A^p(\Ht)$ (cf. Proposition~\ref{prop-ap} below).

We now consider what happens under Bergman projection if $1<p\leq \frac{4}{3}$. For such $p$,  
it turns out that there is {\em absolutely no way} to control the Bergman projection of an $L^p$ function on the Hartogs triangle 
using a weight depending on the distance to the singularity:
\begin{theorem}\label{thm-unbdd} Let $1<p\leq \frac{4}{3}$, and let $\lambda>0$ be a continuous function on 
$(0,1]$. Then $\BpHt$ does not map $L^p(\Ht)$ into $A^p(\Ht,\lambda\circ\delta_1)$.
\end{theorem}
This  pathological phenomenon does not seem to have been noticed before.

\subsection{The Bergman projection} We recall some basic definitions and facts regarding the Bergman projection operator.
 Let  $U\subset\cx^n$ be a bounded domain. The {\em Bergman space} $A^2(U)$ is  defined to be
the intersection $L^2(U)\cap \mathcal{O}(U)$ of the space $L^2(U)$ of square integrable functions on $U$ (with respect to the Lebesgue measure of $\cx^n$) with the space $\mathcal{O}(U)$ of holomorphic functions on $U$. By the Bergman inequality, $A^2(U)$ is a closed subspace of $L^2(U)$. The orthogonal projection operator $\Bp_U:L^2(U)\to A^2(U)$ 
is the {\em Bergman projection} associated with the domain $U$. It follows from the Riesz representation theorem that the Bergman projection is an integral operator with the kernel $\mathbb{B}_U(z,w)$ on $U\times U$ (called the { \em Bergman kernel}), i.e.
$
\Bp_Uf(z)=\int_U\mathbb{B}_U(z,w)f(w)dV(w)
$
for all $f\in L^2(U)$. We refer to \cite{KrantzSCVbook} for the general theory.

It is natural to consider the mapping properties
of $\Bp_U$ on other spaces of functions on $U$, for example $L^p$ spaces or Sobolev spaces.  A survey of mapping properties on  Sobolev spaces on smoothly bounded pseudoconvex domains can be found in \cite{BoasStraube99}. For  results on the regularity of the Bergman projection in $L^p$ space on bounded pseuodoconvex domains
we refer to the following articles and the references therein 
\cite{lanstein, PhongStein77, Barrett84, McNealStein94, KrantzPeloso08, BarrettSahutoglu11, ZeytuncuTran}.
In \cite{ehsani}, estimates were obtained for the Bergman projection on certain non-smooth domains in $L^p$ spaces with  weights which vanish at the singularities of the boundary. This is similar to our Theorem~\ref{thm-bdd}, but the Hartogs triangle is not among the domains to which the results of \cite{ehsani} apply.

\subsection{The Hartogs triangle} The  Hartogs triangle defined in  \eqref{eq-ht}
has remarkable geometric and function-theoretic properties, and is a classical source of counterexamples in complex analysis. The boundary $b\Ht$ of the  domain $\Ht$ has a serious singularity at the point 0, where $b\Ht$ cannot 
even be represented as a graph of a continuous function. The closure $\overline{\Ht}$ does not have a Stein neighborhood basis. The $\dbar$-problem on $\Ht$ is not globally regular, i.e., there is a $\dbar$-closed  $(0,1)$-form  $g\in \mathcal{C}^\infty_{0,1}(\overline{\Ht})$ which is smooth up to the boundary on $\Ht$, 
 such that {\em no solution} $u$ of the equation $\dbar u =g$ lies in $\mathcal{C}^\infty(\overline{\Ht})$ (see \cite{dbarhartogs}).
 
 A standard way of understanding function theory on domains with singular boundary is the use of weights which vanish or blow up at the singular points of the boundary. Consequently, one can obtain estimates which take into account the behavior
 of functions and forms near the singular points. For the Hartogs triangle, this was done in \cite{cshaw1}, where estimates for 
 the canonical solution of the $\dbar$-problem were obtained in weighted Sobolev spaces.
  While the Hartogs triangle has a non-Lipschitz boundary, as a complex manifold it has a 
very simple structure: it is biholomorphic to the product $\Dd$ of  the punctured unit disc $\mathbb{D}^*=\Set{z\in \cx|0<\abs{z}<1}$ and the unit disc $\mathbb{D}$.  Consequently, one can pull back problems on $\Ht$ to problems on $\Dd$, and one gets weights coming from the Jacobian factor. This technique was used in \cite{michelma1992, cshaw1} to study function theory on $\Ht$. Here we use the same method to study the mapping properties of $\BpHt$ in $L^p$ spaces.

This paper is organized as follows. In the following section, we give  a proof of Theorem~\ref{thm-unbdd}, using a
duality argument. After that, in Section~\ref{sec-dd} we consider the Bergman projection on the product $\Dd$ with
respect to a  radial weight on the first factor. In Proposition~\ref{prop-ddsurjective}, we obtain estimates on $\Dd$ using operator-theoretic methods relating to norm-convergence in $L^p$ spaces of Taylor series. In Section~\ref{sec-mainproof}
we prove Theorem~\ref{thm-bdd} using biholomorphic mapping of $\Ht$ with $\Dd$. Finally, in the last section, we 
deduce Corollary~\ref{cor-folk} from Theorem~\ref{thm-bdd} using a description of weighted Bergman spaces on $\Ht$ (see Proposition~\ref{prop-ap}).
 
 \subsection{Acknowledgements} We would like to thank the  American Institute of Mathematics, Palo Alto, for providing a very congenial atmosphere to work on this problem while we were participating in the workshop {\em Cauchy-Riemann Equations in Several Variables} in June 2014. We would also like to thank the anonymous referee for helpful recommendations.

\section{Proof of Theorem~\ref{thm-unbdd}}
For a continuous function $\omega>0$ on a domain $U$, the space $L^p(U,\omega)$ consists of functions $f$ on $U$ for which
\[ \norm{f}_{L^p(U,\omega)}^p = \int_U \abs{f}^p \omega dV<\infty,\]
where $dV$ is Lebesgue measure.
We use the standard duality of $L^p$ spaces to prove the following:
\begin{lemma}\label{lem-duality} Let $U\subset\cx^n$ be a bounded domain, $p>1$, and $\omega>0$ a weight function on $U$. 
Then if $\Bp_U$ is bounded from $L^p(U)$ to $A^p(U, \omega)$, then 
 $\Bp_U$ is  also bounded from $L^q(U, \omega^{1-q})$ to $A^q(U)$, where  $q$ is the conjugate exponent to $p$, i.e. $\frac{1}{p}+\frac{1}{q}=1$.
\end{lemma}
\begin{proof} For $f\in L^q(U,\omega^{(1-q)})\cap L^2(U)$ we estimate
 $\norm{\Bp_U f}_{L^q(U)}$ by using duality of $L^p$ spaces on $U$.   Here and later we denote $\langle u,v \rangle= \int_{\Ht}u\overline{v}dV$ , which if $u,v\in L^2(\Ht)$ is the standard inner product of the Hilbert space
$L^2(\Ht)$. Note that on the bounded domain $U$, we have $L^p(U)\subset L^2(U)$ (for $p\geq2$) by Hölder's inequality.
We let $g\in L^p(U)$ range over the unit sphere $\norm{g}_{L^p(U)}=1$, so that we have
\begin{align*}
\norm{\Bp_U f}_{L^q(U)}&=\sup_{\norm{g}_{L^p(U)}=1} \abs{\left\langle \Bp_U f,g \right\rangle}\\
&=\sup_{\norm{g}_{L^p(U)}=1} \abs{\left\langle f,\Bp_U g \right\rangle} ,\text{ since $\Bp_U$ is self-adjoint}\\
&=\sup_{\norm{g}_{L^p(U)}=1} \abs{\left\langle f\cdot \omega^{-\frac{1}{p}},\left( \Bp_Ug \right)\cdot \omega^{\frac{1}{p}}\right\rangle} \\
&\leq\sup_{\norm{g}_{L^p(U)}=1} \left(\norm{f \cdot \omega^{-\frac{1}{p}}}_{L^q(U)}\cdot \norm{\left(\Bp_Ug\right)\cdot \omega^{\frac{1}{p}} }_{L^p(U)}\right),  \text{ by  H\"older's inequality}\\ 
&=\left(\sup_{\norm{g}_{L^p(U)}=1} \norm{\Bp_U g }_{L^p(U,\omega)}\right)\cdot \norm{f}_{L^q\left(U, \omega^{-\frac{q}{p}}\right)}
%&\leq ||f\delta_1^{\alpha(1-q)/q}||_{q}=||f||_{L^q(\delta_1^{\alpha(1-q)})} \text{ by the assumption}.
\end{align*}
since by assumption the first factor is finite, we have $\left(-\frac{q}{p}\right)=1-q$ and $L^q(U,\omega^{(1-q)})\cap L^2(U)$ is dense in $L^q(U,\omega^{(1-q)})$,  it follows that $\Bp_U$ is bounded from $L^q(U,\omega^{(1-q)})$ to $A^q(U)$.
\end{proof}

\begin{proof}[Proof of Theorem~\ref{thm-unbdd}]
If $1<p\leq\frac{4}{3}$, then the conjugate exponent $q\geq 4$. Therefore, by Lemma~\ref{lem-duality}, to prove the result we only need to show that $\BpHt$ does not map  $L^q(\Ht, (\lambda\circ\delta_1)^{1-q})$ into $A^q(\Ht)$. It therefore suffices to present a function $f\in L^2(\Ht)\cap L^q(\Ht, (\lambda\circ\delta_1)^{1-q})$ such that
$\BpHt f \not \in A^q(\Ht)$.

Let $\chi\geq 0$ be a continuous function on $[0,1]$ such that $\chi\equiv 0$ near 0, and $\chi \equiv 1$ near 1. Let $f$ be the function on $\overline{\Ht}$ given by $f(z_1,z_2)= \left(\chi\circ\delta_1\right)\cdot \overline{z_1},$
which is continuous on $\overline{\Ht}$ and vanishes near the singularity 0, so that $f\in  L^2(\Ht)\cap L^q(\Ht, (\lambda\circ\delta_1)^{1-q})$ for any choice of $\lambda$.  We claim that there is a constant $C>0$ such that  
$\BpHt(f)= C z_1^{-1}$.  Note that the system of monomials $\{z_1^mz_2^n\}$ for 
$m\geq -(n+1)$ and $n\geq 0$ forms a complete orthogonal set in $A^2(\Ht)$, so that it suffices to show that unless $m=-1$ and $n=0$, the function  $f$ is orthogonal 
to each $z_1^mz_2^n$ .  We have, using polar coordinates:
\begin{align*}\langle f, z_1^m z_2^n\rangle &= \int_\Ht \chi(\abs{z_1})\overline{z_1}\cdot\overline{z_1}^m \overline{z_2}^n dV\nonumber\\
&= \int_\Ht \chi(r_1)r_1e^{-i\theta_1}r_1^m e^{-im\theta_1}r_2^n e^{-in\theta_2}r_1r_2dr_1dr_2d\theta_1 d\theta_2\\
&=\left(\int_0^{2\pi}e^{-i(m+1)\theta_1}d\theta_1\right)\cdot \left( \int_0^{2\pi}e^{-in\theta_2}d\theta_2\right)\cdot
\int_\Delta \chi(r_1)r_1^{m+2} r_2^{n+1}dr_1dr_2,
\end{align*}
where $\Delta\subset\rl^2$ is the subset defined by $\{(r_1,r_2): 0\leq r_2<r_1<1\}$.
This integral vanishes unless $m=-1$, $n=0$.

To complete the proof, we verify that $\frac{1}{z_1}$ does not belong to $A^q(\Ht)$, if $q\geq 4$. We have,
\begin{align}\int_\Ht \abs{\frac{1}{z_1}}^qdV&= \int_\Ht\frac{1}{r_1^q}r_1r_2dr_1dr_2d\theta_1d\theta_2\nonumber\\
&= 4\pi^2 \int_{r_1=0}^1r_1^{1-q} \left(\int_{r_2=0}^{r_1}r_2dr_2\right)dr_1\nonumber\\
&=  2\pi^2 \int_0^1r_1^{3-q}dr_1,
\end{align}\label{eq-a4}
which diverges if $q\geq 4$.
\end{proof}

%%%%%%%%%%%%%%%%%%%%%%%%%%%%%%%%%%%%

\section{Weighted $L^p$-regularity of the Bergman projection on $\Dd$.}\label{sec-dd}
We prove the following regularity result for the Bergman projection  on the bidisc. 
As in \eqref{eq-delta1}, we set $\delta_1(z)= \abs{z_1}$.
\begin{proposition}\label{prop-ddsurjective} For $p\geq 2$, the Bergman projection $\BpDd$ is bounded and 
surjective from $L^p(\Dd, \delta_1^{2-p})$ onto $A^p(\Dd)$.
\end{proposition}
Note that  when $p\geq 2$, $A^p(\Dd)$ is identical to the space $A^p(\D^2)$, since each holomorphic function in $L^p(\Dd)$ extends holomorphically  to $\D^2$ (see the proof of Lemma~\ref{lem-ap} below).

For a non-negative integer $N$, we define a map 
$S_N:\mathcal{O}(\D^2)\to \mathcal{O}(\D^2)$ in the following way. For $f\in \mathcal{O}(\D^2)$, with Taylor representation
\begin{equation}\label{eq-f}
f(w_1,w_2)= \sum_{\mu,\nu=0}^\infty a_{\mu,\nu}w_1^\mu w_2^\nu,
\end{equation}
we define the function  $S_Nf$ by setting
\begin{equation}\label{eq-sn}
S_Nf(w_1,w_2)= \sum_{\mu=0}^N\sum_{\nu=0}^\infty a_{\mu,\nu}w_1^\mu w_2^\nu,
\end{equation}
i.e., $S_N$ is the $N$-th partial sum in the $w_1$ variable. 
We have the following analog of a result of Zhu (\cite[Corollary~4]{zhu}):

\begin{lemma}\label{lem-zhu}For $p\geq 1$ and for any $f\in A^p(\D^2)$, as $N$ goes to infinity, $S_N f \to f$
in the $L^p$-norm.

\end{lemma}\begin{proof}
Denote by $\mathbb{T}$ the unit circle $\{\abs{z}=1\}$ in 
the plane and by $\mathbb{T}^2$ the two-torus $\{(z_1,z_2)\in\cx^2| \abs{z_1}=\abs{z_2}=1\}$.
Suppose that $g\in H^p(\mathbb{T}^2)$, the Hardy space on $\mathbb{T}^2$. This means that $g$ 
admits a holomorphic extension to $\D^2$ and has boundary values in $L^p(\mathbb{T}^2)$. For almost all $z_2\in \mathbb{T}$, the function $z_1\mapsto g(z_1,z_2)$ is in $L^p(\mathbb{T})$, and note that the $N$-th partial sum of 
the Fourier series representation of $g(\cdot, z_2)$ is precisely $S_Ng(\cdot, z_2)$. From the classical Riesz theory of 
convergence of Fourier series in $L^p(\mathbb{T})$ (see e.g. \cite[Section~3.5]{grafakos}), we conclude that there is a constant $C_0$, independent of 
the function $g$ such that we have for almost all $z_2\in \mathbb{T}$ that
\begin{equation}\label{eq-riesz}
\int_0^{2\pi}\abs{S_N g(e^{i\theta},z_2)}^p d\theta \leq C_0 \int_0^{2\pi} \abs{g(e^{i\theta},z_2)}^p d\theta.
\end{equation}

Following a standard functional-analytic argument (cf.~\cite[Proposition 1]{zhu}), to prove Lemma~\ref{lem-zhu}, it is sufficient to show that there is a
constant $C>0$ such that $\norm{S_N}\leq C$ for all non-negative integers $N$,
 where we think of $S_N$ as an operator from $A^p(\D^2)$ to itself, and $\norm{\cdot}$ is the operator norm. Now
\begin{align*}
\norm{S_Nf}_{L^p(\D^2)}^p &= \int_\D\int_\D \abs{S_N f(w_1,w_2)}^p dV(w_1) dV(w_2)\\
&= \int_0^1\int_0^1 r_1r_2 \left(\int_0^{2\pi}\int_0^{2\pi}\abs{S_Nf(r_1e^{i\theta_1}, r_2e^{i\theta_2})}^p d\theta_1 d\theta_2 \right) dr_1 dr_2\\
&\leq C_0 \int_0^1\int_0^1 r_1r_2 \left(\int_0^{2\pi}\int_0^{2\pi}\abs{f(r_1e^{i\theta_1}, r_2e^{i\theta_2})}^p d\theta_1 d\theta_2 \right) dr_1 dr_2\\
&= C_0 \norm{f}_{L^p(\D^2)}^p,
\end{align*}
where in the last-but-one line, we have used \eqref{eq-riesz}.
\end{proof}
We now give a sufficient condition for an operator on $A^p(\D^2)$  defined by a certain multiplier sequence on the Taylor coefficients to be bounded. Let $\{t_\mu\}_{\mu=0}^\infty$ be a sequence of complex numbers such that 
\begin{equation}\label{eq-tbv}\sum_{\mu=0}^\infty\abs{t_{\mu+1}-t_\mu}<\infty. \end{equation}
Writing $ t_N = \sum_{\mu=0}^{N-1} (t_{\mu+1}-t_\mu)+t_0,$
and using \eqref{eq-tbv} we see that the limit  $\lim_{N\to \infty} t_N$ exists, and in particular, the sequence $\{t_\mu\}_{\mu=0}^\infty$ is bounded.
It follows that for a holomorphic $f$ on $\D^2$, represented as in \eqref{eq-f}, the new function
\begin{equation}\label{eq-t}
Tf(w_1,w_2)= \sum_{\mu,\nu=0}^\infty t_\mu a_{\mu,\nu}w_1^\mu w_2^\nu
\end{equation}
is also a holomorphic function on $\D^2$. 
\begin{lemma}\label{lem-tbdd}
The map $T$ is bounded from $A^p(\D^2)$ to itself.
\end{lemma}

\begin{proof} Let $f\in A^p(\D^2)$.  We will first show that the partial sums $S_NTf$ of $Tf$ (defined using the operator $S_N$ of \eqref{eq-sn})
converge in $A^p(\D^2)$. Indeed, using summation by parts, we can write
\[ S_N Tf = t_N \cdot S_N f - \sum_{k=0}^{N-1}(t_{k+1}-t_k)S_k f.\]
Note that each of the factors $t_N$ and $S_N f$ in the first term converges to a limit as $N\to\infty$ so that the first term 
converges to a limit in $L^p(\D^2)$ as $N\to \infty$.  Let $C>0$ be such that $\norm{S_N f}_{L^p(\D^2)}\leq C$ for all
$N$. Such $C$ exists since $S_N f\to f$ in $L^p(\D^2)$. Therefore, we see that
\begin{align*} \sum_{k=0}^\infty \norm{ (t_{k+1}-t_k)S_k f}_{L^p(\D^2)}&\leq C \sum_{k=0}^\infty\abs{t_{k+1}-t_k}\\&<\infty.
\end{align*}
This shows the second term is the $(N-1)$-th partial sum of an absolutely convergent series in a Banach space,  and therefore has a limit as $N\to \infty$.

Now let $g=\lim_{N\to\infty}S_N Tf$, where the limit is in the topology of  $A^p(\D^2)$. Let $M\geq 0$ be an integer. By the continuity of  $S_M$ on 
$A^p(\D^2)$ established in the proof of Lemma~\ref{lem-zhu},  it follows that $S_Mg= S_M \left(\lim_{N\to\infty}S_N Tf\right)= \lim_{N\to\infty}S_MS_N Tf= S_M Tf$. Since this is true for each $M$, we have $Tf=g$. Since $g\in A^p(\D^2)$, the operator $T$ maps
$A^p(\D^2)$ boundedly to itself.
\end{proof}
\begin{proof}[Proof of Proposition~\ref{prop-ddsurjective}] We first prove that $\BpDd$ is bounded  from $L^p(\Dd, \delta_1^{2-p})$ to $A^p(\D^2)$ (note that $A^p(\D^2)$ is identical to $A^p(\Dd)$.)  Recall that the Bergman projection on the unit disc is bounded from $L^p(\D)$ to $A^p(\D)$ (see e.g. \cite{zhubook,Forellirudin74}). Using the fact that the Bergman kernel of the product $\D^2$ is the tensor product of the 
Bergman kernels of the factors, by  a simple application of Fubini's theorem it follows that the Bergman projection on the bidisc $\D^2$ is bounded from $L^p(\D^2)$ to $A^p(\D^2)$.
On $\Dd$ we have $0<\delta_1<1$, and therefore for $p\geq 2$, we obtain,
\[ \norm{f}_{L^p(\D^2)}^p =\int_{\D^2} \abs{f}^p dV
\leq \int_{\D^2}\abs{f}^p \delta_1^{2-p}dV
=\norm{f}_{L^p(\D^2,\delta_1^{2-p})}^p,
\]
so that $L^p(\Dd,\delta_1^{2-p})=L^p(\D^2,\delta_1^{2-p})$ is continuously embedded in $L^p(\D^2)$. But as noted above,
the Bergman projection on $\D^2$ is bounded from $L^p(\D^2)$ to $A^p(\D^2)$, so that by composing the two maps, we see that 
$\BpDd$ is bounded from $L^p(\Dd,\delta^{2-p})$ to $A^p(\D^2)=A^p(\Dd)$.

To show that $\BpDd$ surjective onto $A^p(\Dd)$, we construct an operator $U:A^p(\Dd)\to L^p(\Dd,\delta^{2-p})$ such that $\BpDd\circ U$ is the identity on $A^p(\Dd)$, i.e, for $f\in A^p(\Dd)$, the function $Uf$ 
is projected by $\BpDd$ to $f$, which shows the surjectivity of $\BpDd$. 

To construct $U$, in the equation \eqref{eq-t} we let  $t_\mu= 1+(\mu+1)^{-1}$,
which being monotone and bounded certainly satisfies \eqref{eq-tbv}. Further, a computation shows that
\begin{equation}\label{eq-tmu}
\frac{\left\langle \abs{w_1}^2 w_1^\mu, w_1^m\right\rangle_\D}{\left\langle w_2^m, w_2^m\right\rangle_\D} = \frac{1}{t_\mu}\delta_{\mu m}.
\end{equation}
Defining $T$ as in \eqref{eq-t}, we let 
$ Uf = \delta_1^2\cdot Tf.$ Then $U$ maps $A^p(\D^2)$ to $L^p(\Dd, \delta_1^{2-p})$. Indeed:
\begin{align*}
\norm{Uf}^p_{L^p(\Dd, \delta_1^{2-p})}&= \int_\Dd \abs{\delta_1^2 Tf}^p \delta_1^{2-p}dV\\
&= \int_\Dd \delta_1^{2+p}\abs{Tf}^p dV\\
&\leq \int_{\Dd}\abs{Tf}^p dV\\
&\leq C \norm{f}_{L^p(\D^2)}^p<\infty,
\end{align*}
where in the last line we have used Lemma~\ref{lem-tbdd}. 
Further, since $\{w_1^mw_2^n\}_{m,n=0}^\infty$ is an orthogonal set in $L^2(\D^2)$, we see that
\begin{align*} \BpDd(Uf)(w_1,w_2)&= \sum_{m,n=0}^\infty\left\langle\abs{w_1}^2 \sum_{\mu,\nu=0}^\infty t_\mu a_{\mu,\nu}w_1^\mu w_2^\nu, w_1^m w_2^n\right\rangle_\Dd\cdot \frac{w_1^mw_2^n}{\left\langle w_1^m, w_1^m\right\rangle_\D\cdot \left\langle w_2^n , w_2^n\right\rangle_\D}\\
&= \sum_{m,n=0}^\infty \sum_{\mu,\nu=0}^\infty t_\mu a_{\mu,\nu}\left\langle\abs{w_1}^2 w_1^\mu w_2^\nu, w_1^m w_2^n\right\rangle_\Dd\cdot \frac{w_1^mw_2^n}{\left\langle w_1^m, w_1^m\right\rangle_\D\cdot \left\langle w_2^n , w_2^n\right\rangle_\D}\\
&= \sum_{m,n=0}^\infty \sum_{\mu,\nu=0}^\infty t_\mu a_{\mu,\nu}
\frac{\left\langle\abs{w_1}^2 w_1^\mu, w_1^m \right\rangle_\D}{\left\langle w_1^m, w_1^m\right\rangle_\D} \cdot
\frac{\left\langle w_2^\nu,  w_2^n\right\rangle_\D}{\left\langle w_2^n , w_2^n\right\rangle_\D}w_1^mw_2^n\\
&= \sum_{m,n=0}^\infty \sum_{\mu,\nu=0}^\infty t_\mu a_{\mu,\nu}\cdot \frac{1}{t_\mu}\delta_{\mu m}
\delta_{\nu n}w_1^mw_2^n, \quad \text{ using \eqref{eq-tmu}}\\
&=\sum_{m,n=0}^\infty a_{m,n}w_1^mw_2^n=f(w_1,w_2).
\end{align*}
\end{proof}

%%%%%%%%%%%%%%%%%%%%%%%%%%%%%%%%%%%%

\section{Proof of Theorem~\ref{thm-bdd}.}\label{sec-mainproof}

\subsection{Change of Variables}Let $U$ and $V$ be two bounded domains in $\cx^n$, and let $\Phi:U\to V$ be a biholomorphic map. We set 
\[u = \abs{\det \Phi'}
\text{ and }  v = \abs{(\det \Phi')\circ\Phi^{-1}}.\]

We prove the following:

\begin{lemma}\label{lem-2}Let $\alpha$ be a real number. The following are equivalent:
\begin{enumerate}
\item $\Bp_U$ is bounded and surjective from $L^p(U, u^{2-p})$ to 
$A^p(U,u^{2-p+\alpha})$.
\item  $\Bp_V$ is bounded and surjective from  $L^p(V)$ to $A^p(V, v^\alpha)$.
\end{enumerate}
\end{lemma}

For a function $h$  on $V$ define
\begin{equation}\label{eq-tphi}
T_\Phi h = \det(\Phi')\cdot ( h\circ \Phi).
\end{equation}
 Let $\alpha$ be a real number. We first note the following
\begin{lemma}\label{lem-isom}
 The map
\begin{equation}\label{eq-isom1} T_\Phi:L^p\left(V, v^\alpha\right)\to L^p\left(U, u^{2-p+\alpha}\right)\end{equation}
is an isometric isomorphism of Banach spaces.
\end{lemma}
\begin{proof} Indeed,
\begin{align*}
\norm{T_\Phi f}_{L^p(U, u^{2-p+\alpha})}^p 
&= \int_U \abs{(\det\Phi')\cdot (f\circ \Phi)}^p {\abs{\det\Phi'}}^{2-p+\alpha} dV\\
&= \int_U \abs{f\circ \Phi}^p \abs{\det \Phi'}^\alpha \abs{\det \Phi'}^2 dV\\
&= \int_V \abs{f}^p \abs{\det \Phi'\circ \Phi^{-1}}^\alpha dV,
\end{align*} 
where in the last line we use  the change of variables formula.
\end{proof}

\begin{proof}[Proof of Lemma~\ref{lem-2}]  Recall the Bell transformation formula for the Bergman projection (cf. \cite{bell}), which can be written as
\begin{equation}\label{eq-trans} \Bp_U \circ T_\Phi = T_\Phi\circ \Bp_V,
\end{equation}
So that we have $\Bp_V= T_\Phi^{-1}\circ \Bp_U \circ T_\Phi$. Start with 1$\implies$2.  We know that $T_\Phi$ is an isometry from  $L^p(V)$ to $L^p(U, u^{2-p})$, by hypothesis, $\Bp_U$ is bounded from $L^p(U, u^{2-p})$ to 
$A^p(U,u^{2-p+\alpha})$, and we also know that $T_\Phi^{-1}$ is an isometry from $A^p(U,u^{2-p+\alpha})$ to $A^p(V, v^\alpha)$, so that we have (2). The part 2$\implies$1 can be done exactly the same way.
\end{proof}

\subsection{Proof of Theorem~\ref{thm-bdd}}
We apply Lemma~\ref{lem-2}.
Let $\Phi:\Dd\to \Ht$ be given by
\begin{equation}\label{eq-phi}
\Phi(w_1,w_2) = (w_1,w_1w_2),
\end{equation}
so that in Lemma~\ref{lem-2}, we have
\[ u(w_1,w_2) = \abs{w_1}
\text{ and } v(z_1,z_2)= \abs{z_1}, \]
and we take $\alpha=p-2$. Lemma~\ref{lem-2} shows that the following statements are equivalent:
\begin{enumerate}
\item $\BpHt$ is bounded and surjective from $L^p(\Ht)$ to $A^p(\Ht, \delta_1^{p-2})$.
\item $\BpDd$ is bounded and surjective from $L^p(\Dd,\delta_1^{2-p})$ to $A^p(\Dd)$.
\end{enumerate}
Since the second statement was proved in Proposition~\ref{prop-ddsurjective}, the result follows for $p\geq 2$.

%%%%%%%%%%%%%%%%%%%%%%%%%%%%%%%%%%%%%%%%%%%%%

\section{Proof of Corollary~\ref{cor-folk}}
\subsection{Weighted Bergman spaces on the punctured disc}

On the punctured unit disc $\D^*=\set{z\in\cx| 0<\abs{z}<1}$ , denote by $\delta$ the weight function $\delta(z)=\abs{z}$.
\begin{lemma} \label{lem-ap} Let $p \geq 2$  and $-2 < \alpha \leq 0.$ Then we have
\begin{equation}
A^p(\D^*,\delta^\alpha)= A^p(\D),
\end{equation}
and the two norms $\norm{\cdot}_{L^p(\D)}$ and $\norm{\cdot}_{L^p(\D, \delta^\alpha)}$ on this space are equivalent.
\end{lemma}
\begin{proof}
We first claim that each function in  $A^p(\D^*, \delta^\alpha)$ has a {\em removable} singularity at 0.  Since $-2<\alpha \leq 0$, we have
\begin{equation}\label{eq-wtheta} \int_{\D}\abs{w}^\alpha dV(w) = 2\pi \int_0^1 r^{1+\alpha}dr <\infty,\end{equation}
so that $\delta^\alpha\in L^1(\D)$. It follows by Hölder's inequality that $ A^p(\D^*)\subset A^2(\D^*),$
so that to establish the claim it is sufficient to show that each function in $A^2(\D^*)$  extends holomorphically to 0. Now if $f\in \mathcal{O}(\D^*)$, writing $f(w)=\sum_{\nu=-\infty}^\infty a_\nu w^\nu$, we have $ \int_0^{2\pi}\abs{f(re^{i\theta})}^2d\theta=2\pi \sum_{\nu=-\infty}^\infty \abs{a_\nu}^2r^{2\nu}.$
If $f\in A^2(\D^*)$ we have
\begin{align*}
\norm{f}^2_{A^2(\D^*)}&= \int_{0}^1 \left( \int_0^{2\pi}\abs{f\left(re^{i\theta}\right)}^2d\theta\right)rdr\\
&=\sum_{\nu=-\infty}^\infty \abs{a_\nu}^2\int_{0}^1 r^{2\nu+1}dr,
\end{align*}
by the monotone convergence theorem. Since this is finite, we must have $a_\nu=0$ unless $2\nu+1>-1$, i.e., $\nu\geq 0$, which verifies the claim that 
the singularity of $f$ at 0 is removable.

Since every element of $A^p(\D^*)$ has a removable singularity at 0, it follows that  $A^p(\D^*)= A^p(\D).$
Now since $\alpha\leq 0$, we have $\delta^\alpha \geq 1$ on $\D$, so we
have a continuous embedding
\begin{equation}\label{eq-inclusion} A^p(\D^*,\delta^\alpha) \hookrightarrow A^p(\D^*)= A^p(\D). 
\end{equation}
We claim that the injective map in \eqref{eq-inclusion} is in fact surjective. 
Let  $f\in A^p(\D)$. We have
\begin{align*}
\norm{f}^p_{L^p(\D^*,\delta^\alpha)}&= \int_{\D} \abs{f(w)}^p \abs{w}^\alpha dV(w)\\
&= \int_{\abs{w}<\frac{1}{2}}\abs{f(w)}^p \abs{w}^\alpha dV(w)
+ \int_{\abs{w}\geq \frac{1}{2}}\abs{f(w)}^p \abs{w}^\alpha dV(w)\\
&\leq C \int \abs{w}^\alpha dV(w)+ \frac{1}{2^\alpha}\int_{\abs{w}\geq \frac{1}{2}}\abs{f(w)}^p dV(w),
\end{align*}
The first of these integrals is finite (see \eqref{eq-wtheta}) and so is the second one since by hypothesis $f\in L^p(\D)$. Therefore the map in \eqref{eq-inclusion} is a continuous bijection and the result follows by the open mapping theorem.
\end{proof}
We deduce the following  from Lemma~\ref{lem-ap}:
\begin{proposition} \label{prop-ap}
If $2\leq p < 4$ then 
\[ A^p(\Ht, \delta_1^{p-2}) = A^p(\Ht),\]
and the two norms $\norm{\cdot}_{L^p(\Ht)}$ and $\norm{\cdot}_{L^p(\Ht, \delta_1^\alpha)}$ on this space are equivalent.
\end{proposition}
\begin{proof} Let $\delta_1$ be as in \eqref{eq-delta1}. Then by an application of Fubini's theorem to the bidisc and Lemma~\ref{lem-ap}, we see that   $A^p(\D^*\times \D, \delta_1^\alpha)= A^p(\D^2),$  with equivalence of norms, 
provided $p\geq 2$ and $-2<\alpha\leq 0$. In particular, if $2\leq p<4$,  then  $-2<2-p\leq 0$, so that we have
\begin{equation}\label{eq-ap}
 A^p(\D^*\times \D, \delta_1^{2-p})= A^p(\D^2).
\end{equation}
Let the map $\Phi:\Dd\to \Ht$ be given by \eqref{eq-phi}, and let $T_\Phi$ be as in \eqref{eq-tphi}. Noting that $T_\phi$ maps holomorphic 
functions to holomorphic functions, we see from Lemma~\ref{lem-isom} that for each real $\alpha$, $T_\Phi: A^p(\Ht,\delta_1^\alpha)\to A^p(\Dd, \delta_1^{2-p+\alpha})$ is an isometric isomorphism of Banach spaces.
 Letting $\alpha=0$, we see that $T_\Phi: A^p(\Ht)\to A^p(\Dd, \delta_1^{2-p})$ is an isometric isomorphism, and
 letting $\alpha=p-2$, we see that $T_\Phi:A^p(\Ht,\delta_1^{p-2})\to A^p(\Dd)$ is also an isometric isomorphism.  But we already saw that 
 $A^p(\Dd, \delta_1^{2-p})$ and $ A^p(\Dd)$ are the same space of functions on $\Dd$, and the two Banach-space norms are equivalent.
 The result follows.
 \end{proof}
\subsection{Proof of Corollary~\ref{cor-folk}} For part (a), we combine Theorem~\ref{thm-bdd} with Proposition~\ref{prop-ap}. By the former,  if $2\leq p <4$, the operator $\BpHt$ is bounded from $L^p(\Ht)$ to $A^p(\Ht, \delta_1^{p-2})$ and by the latter, the space $A^p(\Ht, \delta_1^{p-2})$ is the same as  $A^p(\Ht)$ as a topological vector space,  therefore the result follows  for $2\leq p <4$.  Applying Lemma~\ref{lem-duality} with $\omega\equiv 1$,  we obtain that $\BpHt$ is bounded (and hence surjective) from $L^p(\Ht)$ to $A^p(\Ht)$ when $\frac{4}{3}<p<2$.

We showed in Section~\ref{sec-dd} above that $\BpHt $ is bounded and {\em surjective} from $L^p(\Ht)$ onto 
$A^p(\Ht,\delta_1^{p-2})$. Consequently to prove part (b) of Corollary~\ref{cor-folk} it suffices to show that, for $p\geq 4$, there is a function $f\in A^p(\Ht,\delta_1^{p-2})$ which is not in $A^p(\Ht)$. Such a function is $f(z_1,z_2)= \frac{1}{z_1}$. The computation leading to \eqref{eq-a4} shows that $f\not\in A^p(\Ht)$ if $p\geq 4$. On the other hand, by a direct computation;
\begin{align*}\norm{f}_{L^p(\Ht,\delta_1^{p-2})}^p &=\int_\Ht \frac{1}{\abs{z_1}^p} \abs{z_1}^{p-2}dV\\
&=4\pi^2 \int_{r_1=0}^{1} \int_{r_2=0}^{r_1} \frac{1}{r_1^2}r_1r_2 dr_1 dr_2\\
&=2\pi^2,
\end{align*}
so that the result follows.

%\end{proof}

\bibliographystyle{alpha}
\bibliography{Lp}

\begin{thebibliography}{MM92}

\bibitem[Bar84]{Barrett84}
David~E. Barrett.
\newblock Irregularity of the {B}ergman projection on a smooth bounded domain
  in {${\bf C}\sp{2}$}.
\newblock {\em Ann. of Math. (2)}, 119(2):431--436, 1984.

\bibitem[Bel81]{bell}
Steven~R. Bell.
\newblock Proper holomorphic mappings and the {B}ergman projection.
\newblock {\em Duke Math. J.}, 48(1):167--175, 1981.

\bibitem[BS99]{BoasStraube99}
Harold~P. Boas and Emil~J. Straube.
\newblock Global regularity of the {$\overline\partial$}-{N}eumann problem: a
  survey of the {$L^2$}-{S}obolev theory.
\newblock In {\em Several complex variables ({B}erkeley, {CA}, 1995--1996)},
  volume~37 of {\em Math. Sci. Res. Inst. Publ.}, pages 79--111. Cambridge
  Univ. Press, Cambridge, 1999.

\bibitem[B{\c{S}}12]{BarrettSahutoglu11}
David Barrett and S{\"o}nmez {\c{S}}ahuto{\u{g}}lu.
\newblock Irregularity of the {B}ergman projection on worm domains in {$\Bbb
  C^n$}.
\newblock {\em Michigan Math. J.}, 61(1):187--198, 2012.

\bibitem[CC91]{dbarhartogs}
J.~Chaumat and A.-M. Chollet.
\newblock R\'egularit\'e h\"old\'erienne de l'op\'erateur {$\overline\partial$}
  sur le triangle de {H}artogs.
\newblock {\em Ann. Inst. Fourier (Grenoble)}, 41(4):867--882, 1991.

\bibitem[Che13]{Chen13}
Liwei Chen.
\newblock The {$L^p$} boundedness of the {B}ergman projection for a class of
  bounded {H}artogs domains.
\newblock {\em arXiv:1304.7898}, 2013.

\bibitem[CS13]{cshaw1}
Debraj Chakrabarti and Mei-Chi Shaw.
\newblock Sobolev regularity of the {$\overline{\partial}$}-equation on the
  {H}artogs triangle.
\newblock {\em Math. Ann.}, 356(1):241--258, 2013.

\bibitem[EL08]{ehsani}
Dariush Ehsani and Ingo Lieb.
\newblock {$L^p$}-estimates for the {B}ergman projection on strictly
  pseudoconvex non-smooth domains.
\newblock {\em Math. Nachr.}, 281(7):916--929, 2008.

\bibitem[FR75]{Forellirudin74}
Frank Forelli and Walter Rudin.
\newblock Projections on spaces of holomorphic functions in balls.
\newblock {\em Indiana Univ. Math. J.}, 24:593--602, 1974/75.

\bibitem[Gra08]{grafakos}
Loukas Grafakos.
\newblock {\em Classical {F}ourier analysis}, volume 249 of {\em Graduate Texts
  in Mathematics}.
\newblock Springer, New York, second edition, 2008.

\bibitem[KP08]{KrantzPeloso08}
Steven~G. Krantz and Marco~M. Peloso.
\newblock Analysis and geometry on worm domains.
\newblock {\em J. Geom. Anal.}, 18(2):478--510, 2008.

\bibitem[Kra01]{KrantzSCVbook}
Steven~G. Krantz.
\newblock {\em Function theory of several complex variables}.
\newblock AMS Chelsea Publishing, Providence, RI, 2001.
\newblock Reprint of the 1992 edition.

\bibitem[LS12]{lanstein}
Loredana Lanzani and Elias~M. Stein.
\newblock The {B}ergman projection in {$L^p$} for domains with minimal
  smoothness.
\newblock {\em Illinois J. Math.}, 56(1):127--154 (2013), 2012.

\bibitem[MM92]{michelma1992}
Lan Ma and Joachim Michel.
\newblock {$\mathcal{C}^{k+a}$}-estimates for the
  {$\overline\partial$}-equation on the {H}artogs triangle.
\newblock {\em Math. Ann.}, 294(4):661--675, 1992.

\bibitem[MS94]{McNealStein94}
J.~D. McNeal and E.~M. Stein.
\newblock Mapping properties of the {B}ergman projection on convex domains of
  finite type.
\newblock {\em Duke Math. J.}, 73(1):177--199, 1994.

\bibitem[PS77]{PhongStein77}
D.~H. Phong and E.~M. Stein.
\newblock Estimates for the {B}ergman and {S}zeg\"o projections on strongly
  pseudo-convex domains.
\newblock {\em Duke Math. J.}, 44(3):695--704, 1977.

\bibitem[Zey13]{ZeytuncuTran}
Yunus~E. Zeytuncu.
\newblock {$L^p$} regularity of weighted {B}ergman projections.
\newblock {\em Trans. Amer. Math. Soc.}, 365(6):2959--2976, 2013.

\bibitem[Zhu91]{zhu}
Ke~He Zhu.
\newblock Duality of {B}loch spaces and norm convergence of {T}aylor series.
\newblock {\em Michigan Math. J.}, 38(1):89--101, 1991.

\bibitem[Zhu07]{zhubook}
Kehe Zhu.
\newblock {\em Operator theory in function spaces}, volume 138 of {\em
  Mathematical Surveys and Monographs}.
\newblock American Mathematical Society, Providence, RI, second edition, 2007.

\end{thebibliography}

% ----------------------------------------------------------------
\end{document}